\documentclass[12pt,a4paper,reqno]{amsart}
\usepackage{amsfonts}
\usepackage{amsthm}
\usepackage[normalem]{ulem}
\usepackage{amsmath,amsfonts,amssymb,amsthm,accents}
\usepackage{mathtools}

\usepackage{mathrsfs}
\usepackage{array} 
\usepackage{comment}
\usepackage{indentfirst}
\allowdisplaybreaks
\usepackage{setspace}
\usepackage{amscd}
\usepackage[latin2]{inputenc}

\usepackage{enumerate}
\usepackage[mathscr]{eucal}
\usepackage{enumitem}
\usepackage{bbm}
\usepackage[margin=2.3cm]{geometry}
\usepackage[bookmarks=false]{hyperref}
\hypersetup{
	colorlinks=true,
	linkcolor=blue,
	filecolor=blue,
	citecolor = blue,      
	urlcolor=cyan,
}

\theoremstyle{plain}
\newtheorem{Th}{Theorem}[section]

\newtheorem{Theorem}[Th]{Theorem}

\theoremstyle{definition}
\newtheorem{Lemma}[Th]{Lemma}
\newtheorem{Cor}[Th]{Corollary}

\newtheorem{Rem}[Th]{Remark}
\newtheorem{?}[Th]{Problem}

\newtheorem{Result}[Th]{Result}

\theoremstyle{definition}
\newtheorem{Fact}[Th]{Fact}

\theoremstyle{remark}
\newtheorem*{notations}{Notations}

\newcolumntype{x}[1]{%
	>{\centering\hspace{0pt}}m{#1}}%

\newcommand{\tnhl}{\tabularnewline\hline}

\newcommand{\ra}{\rightarrow}

\newcommand{\e}{\mathbb{E}}
\newcommand{\inte}{[0,\delta]}
\newcommand{\p}{\mathbb{P}}
\newcommand{\re}{{\mathbb{R}}}
\newcommand{\z}{\mathbb{Z}}

\newcommand{\ga}{\gamma}

\newcommand{\st}{\subseteq}

\newcommand{\T}{\mathscr{T}}

\newcommand{\lb}{\left(}
\newcommand{\rb}{\right)}
\newcommand{\nat}{\mathbb{N}}
\newcommand{\ep}{\epsilon}
\newcommand{\tm}{\textcolor{magenta}}
\newcommand{\nt}{\mathscr{N}_T}
\newcommand{\lt}{\mathscr{L}_T}
\newcommand{\deltat}{T}
\newcommand{\nn}{\mathscr{N}}
\newcommand{\murad}{\mu_{r}}
\newcommand{\lnn}{L_{n}^{1/n}}

\begin{document}
	\onehalfspacing
	\title[Overcrowding estimates]{Overcrowding Estimates for zero count and nodal length of stationary Gaussian processes}
	\author{Lakshmi Priya}
	\address{Department of Mathematics, Indian Institute of Science, Bangalore 560012, India}
	\email{lakshmip@iisc.ac.in}
	
	\thanks{This work is supported  by CSIR-SPM fellowship (File No. SPM-07/079(0260)/2017-EMR-I)), CSIR, Government of India and by a UGC CAS-II grant (Grant No. F.510/25/CAS-II/2018(SAP-I))}
	\begin{abstract}
		{Assuming certain conditions on the spectral measures of centered stationary Gaussian processes on $\re$ (or $\re^2$), we show that the probability of the event that their zero count  in an interval (resp., nodal length in a square domain) is larger than $n$, where $n$ is much larger than the expected value of the zero count in that interval (resp., nodal length in that square domain), is exponentially small in $n$. }
		\end{abstract}
	\maketitle
	\section{Introduction}
\subsection{Main Results}
	In this paper we consider centered stationary Gaussian processes on $\re$/$\re^2$, study the unlikely event that there is an excessive zero count/nodal length in a region and obtain probability estimates for them.  
Let $g:\re^2 \ra \re$ be a $C^1$ function such that for every $z\in \re^2$ we have $|g(z)|+|\nabla g(z)| \neq 0$. Then its zero/nodal set  $\mathcal{Z}(g) := g^{-1}\{0\}$ is a smooth one-dimensional  submanifold. Let $f:\re \ra \re$ be any function and $T>0$, we define $\nt(f)$ and $\lt(g)$ to be the zero count of $f$ in $[0,T]$ and the nodal length of $g$ in $[0,T]^2$ respectively, that is 
\begin{align*}
\nt(f) &:= \#\{x\in [0,T]: f(x) = 0\},\\
\lt(g)&:= \text{Length}\{z\in [0,T]^2: g(z)=0\}.
\end{align*}	
When there is no ambiguity of the function under consideration, we simply use $\nt$ and $\lt$. 

Let $X$ be a centered stationary Gaussian process on $\re$/$\re^2$ whose spectral measure $\mu$ is a symmetric Borel probability measure. Stationarity of $X$ implies that $\e[\nt] \propto T$ and $\e[\lt] \propto T^2$, under the assumption that $\mu$ has a finite second moment these expectations are also finite. Under some assumptions on $\mu$, we get probability estimates for the unlikely events  $\nt \gg \e[\nt]$ and $\lt \gg \e[\lt]$. The following are the assumptions we make on $\mu$ and their significance.
\begin{itemize}[align=left,leftmargin=*,widest={7}]
	\item \textit{All moments $C_n$\eqref{defns}/$C_{m,n}$\eqref{defns2} of $\mu$ are finite}: The higher order moments $C_n$ of $\mu$ control the  higher order derivatives of $X$,  hence a control on $C_n$ gives a control on the \textit{oscillations} of $X$, which in turn controls the number of $0$-crossings of $X$. 
	\vskip .2cm
	\item \eqref{a1}/\eqref{a2}: The event $\nt \gg \e[\nt]$/$\lt \gg \e[\lt]$ is intimately related to the event that $X$ takes \textit{small} values in a region (small ball event). Under assumption \eqref{a1}/\eqref{a2}, the small ball event becomes highly unlikely and it is possible to estimate their probabilities.  
	\end{itemize} 
Lemmas \ref{resmain} and \ref{newlenlem} are the quantitative versions of the above statements.

\begin{notations}Before stating the main results, we  introduce some useful notations.
	
	\begin{enumerate}[label={\arabic*.},align=left,leftmargin=*,widest={7}]
		\item Let $\mu$ be a positive Borel measure on $\re$. For $n \in \nat$, we define $C_n$ and $D_n$  by
		
		\begin{equation}\label{defns}
		\begin{gathered}
		C_n := \int_{\re} |x|^n d\mu(x)\text{ and }D_n:= \max\{1, \sqrt{C_{2n}}, \sqrt{C_{2n+2}}\}.
		\end{gathered}
		\end{equation}
		\item For $\mu$ a positive Borel measure on $\re^2$, we denote by $\murad$ the push forward of $\mu$ by the map $\phi$, where $\phi: \re^2 \ra [0,\infty)$ is defined by $\phi(x,y) := \sqrt{x^2 +y^2}$. For $m,n \in \nat \cup \{0\}$, we define the following quantities
		\begin{equation}\label{defns2}
		\begin{gathered}
		C_{m,n} := \int_{\re^2} |x|^m |y|^n d\mu(x,y)\text{ and } \widetilde{R}_n := \max\{\sqrt{C_{k,2n-k}}: 0 \leq k \leq 2n\},\\
		R_n := \max\{1,\widetilde{R}_1, \widetilde{R}_n, \widetilde{R}_{n+1}\}.\\
		\widetilde{L}_n := \lb \int_{0}^{\infty} t^{2n} d\murad(t)\rb^{1/2} \text{ and }
		L_n := \max\{1,\widetilde{L}_1, \widetilde{L}_n, \widetilde{L}_{n+1}\}.
			\end{gathered}
		\end{equation}
		The relation between the moments $L_n$ and $R_n$ is discussed in Remark \ref{remasymp}.
		\vskip .2cm
	\item Let $\mu$ be a positive Borel measure on $\re$/$\re^2$. Then $\mu$ is said to satisfy assumption $(A1)$/ $(A2)$ if: 
		\begin{equation}
		\parbox{.85\textwidth}{$\mu$ is a symmetric Borel probability measure on $\re$ which has a nontrivial absolutely continuous part w.r.t. the Lebesgue measure on $\re$. }\tag{$A1$} \label{a1} 
		\end{equation}
		\begin{equation}
		\parbox{.85\textwidth}{$\mu$ is a symmetric Borel probability measure on $\re^2$  such that there exist $v_1$, $v_2 \in \mathbb{S}^1$ satisfying $v_1 \perp v_2$ and the marginals of $\mu$ on $\re v_1$ and $\re v_2$ satisfy \eqref{a1}.}\label{a2} \tag{$A2$}
		\end{equation}
		\end{enumerate}
\end{notations}
\vskip .2cm

\begin{Theorem} \label{thmonedim} Let $X$ be a centered stationary Gaussian process on $\re$ whose  spectral measure $\mu$ is such that it satisfies \eqref{a1} and all its moments $C_n$ are finite. Then there exist constants $c>0$, $b, B \in (0,1)$ such that for every $\ep \in (0,1/2)$, $n \in \nat$ such that $n \geq 1/\ep^2$ and {$T \in (0,b \lfloor \ep n \rfloor)$}
	satisfying $(B/D_{n}^{1/n})\cdot (n/\deltat)^{1-2\ep} \geq e$, we have
	\begin{align}\label{mainestimate0}
	\exp(-n^2 \log(cn/T)) \leq \p(\nt \geq n) \leq 2 \exp \left\{ -\frac{\ep n^2}{2} \log\left(\frac{B}{D_{n}^{1/n}} \lb \frac{n}{\deltat} \rb^{1-2\ep} \right) \right\}.
	\end{align}
\end{Theorem}

\begin{Theorem} \label{thmtwodim} Let $X$ be a centered stationary Gaussian process on $\re^2$ whose spectral measure $\mu$ is such that it satisfies \eqref{a2} and  all its moments $C_{m,n}$ (or equivalently $\widetilde{L}_n$) are finite. Then there are constants $b,B \in (0,1)$ such that for every $\ep \in (0,1/4)$, $n \geq 1/\ep^2$ and {$\deltat \in (0, b\lfloor \ep n \rfloor)$} such that $B/L_{n}^{1/n} \cdot (n/T)^{1-4\ep} \geq e$, we have
	\begin{align}\label{mainest3}
	\p(\lt > 4 n \deltat)  \leq 6 \exp \left\{-\frac{\ep n^2}{2} \log\left(\frac{B}{L_{n}^{1/n}} \lb \frac{n}{\deltat} \rb^{1-4\ep} \right) \right\}.
	\end{align}
\end{Theorem}

\begin{Rem} 
	If we fix $n \geq 16$ and take $\ep = 1/4$ in \eqref{mainestimate0}, we can conclude that there are constants $b_n, b'_n >0$ such that for every $\delta \in (0,b_n')$ we have
	\begin{align*}
	\p(\nn_{\delta} \geq  n) \leq b_n~ \delta^{n^2/16},
	\end{align*}
	and this indicates short-range repulsion of the zeros of $X$.
\end{Rem}
In Section \ref{seccons}, we apply Theorems \ref{thmonedim} and \ref{thmtwodim} to specific classes of spectral measures $\mu$ and get tail estimates for $\nt$ and $\lt$;  these results are summarized in the following tables.
\begin{table}[ht]
	\renewcommand{\arraystretch}{2.6}
	\begin{tabular}{ |x{2.64cm}|x{3.3cm}|x{3.2cm}|x{2.7cm}|x{2.3cm}|}
		\hline
		\bf{Growth of $C_n$}& \bf{Example of $\mu$} &  \bf{Constraints on $T$ and $n$} & \bf{ $\log \p(\nt \geq n)$}& \bf{$(\e[\nt^m])^{\frac{1}{m}} \lesssim $}\tnhl 
		$C_n \leq q^n$, for $q>0$&  Any $\mu$ with supp$(\mu) \subseteq [-q,q]$ &  $T\geq 1,~n \geq CT$, for $C\gg 1$& $\asymp - n^2 \log(\frac{n}{T})$& $T \vee \sqrt{m}$\tnhl
		$C_n \leq n^{\alpha n}$, for $\alpha \in (0,1)$ & $\mu \sim \mathcal{N}(0,1)$, with $\alpha = 1/2$&$T\geq 1,~ n \geq T^{1/\kappa}$, for $\kappa \in (0,1-\alpha)$& $\asymp -n^2 \log n$& $T^{1/\kappa} \vee \sqrt{m}$ \tnhl
		$C_n \leq n^{\alpha n}$, for $\alpha \geq 1$	&$d\mu(x) \propto e^{-|x|^{\frac{1}{\alpha}}}dx$ &$T=1, n\in\nat$& $\lesssim -n^{\frac{2}{\alpha + \kappa}}$, for any $\kappa >0$& $ m^{(\alpha + \kappa)/2}$ \tnhl
		$\log C_n \lesssim n^{\frac{\gamma +1}{\gamma}}$, for $\gamma >1/2$& $ d\mu(x) \propto e^{-(\log |x|)^{1+\gamma}} \mathbbm{1}_{|x| \geq 1}dx$& $T=1, n\in \nat$ &  $\lesssim -(\log n)^{2\ga}$ & $\exp({c}  m^{\frac{1}{2\ga -1}})$\tnhl
	\end{tabular}
	\vskip .5cm
	\caption{Consequences of Theorem \ref{thmonedim} } \label{table1}
\end{table}

\begin{table}[ht]
	\renewcommand{\arraystretch}{2.6}
	\begin{tabular}{ |x{2.65cm}|x{3.3cm}|x{3.2cm}|x{2.7cm}|x{2.5cm}|}
		\hline
		\bf{Growth of $L_n$}& \bf{Example of $\mu$ or $X$} &  \bf{Constraints on $T$ and $\ell$} & \bf{ $\log \p(\lt \geq \ell)$}& \bf{$(\e[\lt^m] )^{\frac{1}{m}} \lesssim $}\tnhl 
		$L_n \leq q^n$, for $q>0$&  Random plane wave $(d\mu(\theta) \sim$ $d\theta/2\pi$  on $\mathbb{S}^1)$ &  $T\geq 1,~\ell \geq CT^2$, for $C\gg 1$& $\lesssim -(\frac{\ell^2}{T^2} \log \frac{\ell}{T^2})$& $T(T \vee \sqrt{m})$\tnhl
		$L_n \leq n^{\alpha n}$, for $\alpha \in (0,1)$ & $\mu \sim \mathcal{N}(0,I)$, with $\alpha = 1/2$&$T\gg 1,~ \ell \geq T^{\frac{\kappa +1}{\kappa}}$, for $\kappa \in (0,1-\alpha)$& $\lesssim - (\frac{\ell^2}{T^2} \log \ell )$ & $T(T^\frac{1}{\kappa} \vee \sqrt{m})$ \tnhl
		$L_n \leq n^{\alpha n}$, for $\alpha \geq 1$	&$d\mu_{r}(x) \propto e^{-x^{\frac{1}{\alpha}}}~\mathbbm{1}_{x>0}~dx$ &$T=1, \ell \gg 1$& $\lesssim - \ell^{\frac{2}{\alpha + \kappa}}$, for any $\kappa >0$& $ m^{(\alpha + \kappa)/2}$ \tnhl
		$\log L_n \lesssim n^{\frac{\ga + 1}{\gamma}}$, for $\gamma >1/2$& $ d\mu_{r}(x) \propto e^{-(\log x)^{1+\gamma}}~ \mathbbm{1}_{x \geq 1}~dx$& $T=1, \ell \gg 1$ &  $\lesssim -(\log \ell )^{2\ga}$ & $\exp({c} m^{\frac{1}{2\ga -1}})$\tnhl
	\end{tabular}
	\vskip .5cm
	\caption{Consequences of Theorem \ref{thmtwodim}} \label{table2}
\end{table}

\subsection{Prior work $\&$ comparison with our results} Before we comment about how our results compare with what is already known, let us  briefly review the results known about zero count and nodal length of stationary Gaussian processes.  A more detailed and comprehensive account of results known so far can be found in  \cite{July2020,bdfz,hugo}.

As before, let $X$ be a centered stationary Gaussian processes on $\re$/$\re^2$ whose spectral measure $\mu$ is a symmetric  Borel probability measure. We denote its  covariance kernel by $k = \widehat{\mu}$. The exact value of the expectations of $\nt$ and $\lt$ are given by the Kac-Rice formulas. Under the assumption that the tails of $\mu$ are light enough and some integrablility assumptions on $k$ and $k''$, it was shown in \cite{JCuzik1} that Var$\nt \asymp T$ and a central limit theorem was also established for $\nt$. Finiteness of moments of $\nt$ and bounds for these were established in \cite{  A2019,AW,NW} under the assumption that the moments of $\mu$ are finite and more recently the asymptotics of the central moments of $\nt$ were obtained in \cite{July2020} assuming certain decay of $k$. Under the assumption that $\mu$ has very light tails and $k$ is integrable, it was shown in \cite{bdfz} that $\nt$ concentrates exponentially around its mean. Central limit theorems and variance asymptotics for $\lt$ in specific examples of stationary Gaussian processes were obtained in \cite{Kratz, npr}.
Under the assumption that moments of $\mu$ are finite, finiteness of moments of $\lt$ were established in \cite{azaisetal}.


These results can be broadly classified into two categories based on the assumptions made on their spectral measure or covariance.
\begin{itemize}[align=left,leftmargin=*,widest={7}]
	\item Decay of covariance $k$ (and light tails of $\mu$): The philosophy behind such an assumption is that decay of $k$ implies \textit{quasi-independence} of events in well separated regions and hence for a large $T$, $\nt$ is approximately a sum of  identically distributed random variables which are  $M$-dependent. Although it is quite easy to state this idea, building on it and making it work is far from trivial. Hence the concentration results in \cite{bdfz,JCuzik1} can be viewed as stemming from some underlying independence. 
	\item Finiteness of moments of $\mu$: In the other results like \cite{  A2019,AW,NW} which established finiteness of moments of $\nt$, the only assumption made on $\mu$ is finiteness of its moments. The significance of this, as discussed earlier, is that a control on the moments imply a control on the oscillations and hence on the zero count also. 
	\end{itemize}

The nature of the results from these two categories also differ: in the former, the assumptions made are strong and so are the conclusions (concentration) and in the latter, the assumptions are much weaker and the conclusions (non-explicit moment bounds) are also weak.
 In terms of the assumptions made on the spectral measure/covariance, our result falls in the second category, but surprisingly we get quite strong conclusions. To illustrate this claim, we present some consequences of  Theorem \ref{thmonedim}; the following are  two instances where we get the exact asymptotics for the deviation probabilities $\p(\nt \geq n)$:
	\begin{itemize}[align=left,leftmargin=*,widest={7}]
	\item If $\mu$ is compactly supported, there is $C > 1$  such that if $T>1$ and $CT \leq n$, then
	\begin{align*}
	-\log \p(\nt \geq n) \asymp n^2 \log(n/T).
	\end{align*}
	\item If $\mu$ is such that $C_n \leq n^{\alpha n}$ for some $\alpha \in (0,1)$, then for every $\kappa \in (0,1-\alpha)$ and $n,T$ such that $1 \leq T \leq n^{\kappa}$, we have
	\begin{align*}
	-\log \p(\nt \geq n) \asymp n^2 \log n.
	\end{align*}
\end{itemize}
It is also interesting to compare the exponential concentration result (\cite{bdfz}, Theorem 1.1)  with Theorem \ref{thmonedim}; both the results give estimates for $\p(|\nt - \e[\nt]| \geq F(T))$. The former gives estimates when the fluctuation $F(T)$ is of order $T$, while our result gives estimates when $F(T)/T \gg 1$ and  $F(T)$ is large (how large depends on the growth of moments $C_n$) enough. Hence only in certain situations can we actually compare these two results and one such instance is when $\mu$ is compactly supported and has a density which belongs to $ W^{1,2}$. There exists $C > 1$ (depending on $\mu$) such that for $\eta >0$ and $F(T) = \eta T$, the estimates for the deviation probability from \cite{bdfz} and Theorem \ref{thmonedim} are $\exp(-c_{\eta}T)$, $\forall \eta >0$ and $\exp(-c\eta^2 T^2)$, $\forall \eta \geq C$  respectively. So this naturally raises a question about the true deviation probabilities: either the true deviations for $\eta \in (0,C)$ are much smaller than $\exp(-c_{\eta} T)$ or there is some $c_0 \in (0,C)$ such that  the behaviour of the deviation probabilities are different  for $\eta \in (0,c_0)$ and $\eta \in (c_0,C)$ which hints at a possible JLM law.

\subsection{Plan of the paper}In Section \ref{secidea}, we present the main ideas in the proofs of Theorems \ref{thmonedim} and \ref{thmtwodim} and also give a sketch of their proofs. The relevant details and calculations required to complete the proofs are presented in Sections \ref{seconedim} and \ref{sectwodim}. In Section \ref{secprelim}, we recall some known results and present other preliminary results required in proving our main theorems.  In Section \ref{seccons}, we deduce the tail bounds for $\nt$ and $\lt$ presented in Tables \ref{table1} and \ref{table2} from Theorems \ref{thmonedim} and \ref{thmtwodim} respectively.  

\section{Idea of the proof}\label{secidea}
In this section, we present the key ideas in the proofs of  Theorems \ref{thmonedim} and \ref{thmtwodim}; we also give a brief sketch of their proofs here. The calculations and detailed arguments will be presented in Sections \ref{seconedim} and \ref{sectwodim}. 
	\subsection{ Proof idea  {of} Theorem \ref{thmonedim}}\label{secideaone}
We now present a simple yet very useful idea regarding the zero count of a smooth function and this serves as the starting point for proving the upper bound in Theorem \ref{thmonedim}. This idea has  previously been employed in \cite{A2019,AW,NW} to study moments of the  zero count of random functions. {The following lemma is inspired by these earlier {versions}  and is a slight modification of those}.


\begin{Lemma} \label{resmain} Let $T, M >0$ and $n \in \nat$. Let $f:[0, 2T] \ra \re$ be a  smooth function which  has at least $n$ distinct roots in $[0,T]$ and  $\| f^{(n)}\|_{L^{\infty}[0,2T]}~ \leq M$. Then for every $0 \leq k \leq n$, we have
	\begin{align*}
	\|f^{(n-k)}\|_{L^{\infty}[T,2T]} \leq M \frac{(2T)^k}{k!}.
	\end{align*}
\end{Lemma}
\begin{proof}
	Let  $g:[0,T] \ra \re$ be a smooth function and let $x_1 < x_2$ be two zeros of $g$, then there exists a zero of $g'$ in $[x_1,x_2]$. This along with the fact that $f$ has at least $n$ distinct roots in $[0,T]$ implies that there are points $0\leq \alpha_{n-1} \leq \alpha_{n-2}\leq \cdots \leq \alpha_1 \leq \alpha_0 \leq T$ such that for every $0 \leq k \leq n-1$, we have $f^{(k)} (\alpha_k) = 0$.  For $t \in [\alpha_{n-1}, 2T]$, we have
	\begin{equation} \label{eqnew1}
	\begin{gathered}
	|f^{(n-1)}(t)|  = |\int_{\alpha_{n-1}}^{t} f^{(n)}(x)~ dx|  \leq M(t-\alpha_{n-1}) \leq Mt.
	\end{gathered}
	\end{equation}
	Using the estimate for $|f^{(n-1)}|$ in \eqref{eqnew1}, we have for every $t \in [\alpha_{n-2}, 2T]$
	\begin{align*}
	|	f^{(n-2)}(t)| = |\int_{\alpha_{n-2}}^{t} f^{(n-1)}(x)~ dx | \leq M \lb \frac{t^2 - \alpha_{n-2}^{2}}{2} \rb  \leq M \frac{t^2}{2}. 
	\end{align*}
	Thus inductively we can establish that for every $k \leq n-1$ and every $t \in [\alpha_{n-k}, 2T]$
	\begin{align*}
	|f^{(n-k)}(t)| \leq M \frac{t^k}{k!} \leq M \frac{(2T)^k}{k!},
	\end{align*}
	and this establishes our claim. 
\end{proof}
\noindent We get the following result as  an immediate consequence of Lemma \ref{resmain}.
\begin{Lemma} \label{lem2_onedim} Let $F: \re \ra \re$ be a smooth random function, then for every $T,M>0$ and $n \in \nat$ we have
		\begin{equation}
	\p(\nt \geq n) \leq \p \lb \lVert F \rVert_{L^{\infty}[\deltat, 2\deltat]} \leq M \frac{(2\deltat)^n}{n!} \rb + \p(\lVert F^{(n)} \rVert_{L^{\infty}[0,2\deltat]} > M). \label{main_onedim}
	\end{equation}
	\end{Lemma}
\subsubsection*{Upper bound in Theorem \ref{thmonedim}}
We use Lemma \ref{lem2_onedim} to prove the upper bound in Theorem \ref{thmonedim} and hence we only need to estimate the two terms on the r.h.s. of \eqref{main_onedim}. The second term is estimated using  well known tail bounds for the supremum of Gaussian processes; we recall these standard results in Section \ref{secsupGP}. Since we are working with stationary Gaussian processes, calculating the metric entropy and the corresponding Dudley integral are quite straightforward.

 The first term corresponds to a small ball event and the discussion in Section \ref{secsmallball} is devoted to obtaining probability estimates for such small ball events. We briefly explain how this is done: say $\eta >0$ and  we want to estimate $\p(\|X\|_{L^{\infty}[0,T]} \leq \eta)$. For $m \in \nat$ and $ 0 \leq k \leq m$, we define $t_{k} := kT/m$ which are $(m+1)$ equispaced points in $[0,T]$. Consider the Gaussian vector $V_m$ defined by 
$V_m := (X_{t_{0}}, \ldots, X_{t_{m}}).$
Let $\Sigma_m$ and $\phi_m$ be the covariance matrix and the density of $V_m$ respectively, then 
\begin{align*}
\phi_m(x) = \frac{1}{(\sqrt{2\pi})^{m+1}~ |\Sigma_m|^{1/2}} \exp \lb -\frac{\langle \Sigma_{m}^{-1}x,x \rangle}{2}\rb,\\
\text{hence }\|\phi_m\|_{L^{\infty}(\re^{m+1})} \leq \frac{1}{(\sqrt{2\pi})^{m+1}~ |\Sigma_{m}|^{1/2}} \leq \frac{1}{(\sqrt{2\pi})^{m+1}~ \lambda_{m}^{m+1/2}},
\end{align*}
where $\lambda_m$ is the smallest eigenvalue of $\Sigma_m$ and hence the following is an estimate for the  small ball probability
\begin{align}\label{ssbbpp}
\p(\|X\|_{L^{\infty}[0,T]} \leq \eta) \leq \int_{[-\eta,\eta]^{m+1} } \phi_{m}(x) dx \leq \lb \frac{2 \eta}{\sqrt{2\pi}~ \lambda_{m}^{1/2}}\rb^{m+1}.
\end{align}
Result \ref{l2} gives a lower bound for $\lambda_m$ and hence an upper bound for the small ball probability in \eqref{ssbbpp}. We then use \eqref{ssbbpp} with $\eta = M(2T)^n/n!$ and an optimal choice of $m$ to get an esimate for the first term in \eqref{main_onedim}. 

For given values of $n$ and $T$, we need to make an optimal choice of $M$ (it should be small enough so that  the first event is indeed a small ball event and large enough so that the second event is unlikely) so that the r.h.s. of \eqref{main_onedim} is as small as possible.
\subsubsection*{Lower bound in Theorem \ref{thmonedim}} We use the same notations as above and take $m=n$. {If the sign of a continuous function defined on $[0,T]$ alternates at the points $t_k$, then there is necessarily a 0-crossing of the function between every two successive points $t_k$ and hence the zero count in $[0,T]$ is at least $n$}. Thus we have
\begin{align}\label{eq 89}
\p(X_{t_0}<0, X_{t_1}>0, \ldots, (-1)^{n+1} X_{t_{n}} >0 ) \leq \p(\nt \geq n),
\end{align}
and hence to get a lower bound for the above probability, we need a lower bound on the density $\phi_n$ and we get this below. $\Sigma_n$ being a Gram-matrix with all diagonal entries equal to 1, we have $|\Sigma_n| \leq 1$ (\cite{MK}, Lemma 3) and hence 
\begin{align}\label{eq 90}
\phi_n (x) \geq \frac{1}{(\sqrt{2\pi})^{n+1}} \exp \lb -\frac{\langle \Sigma_{n}^{-1}x,x \rangle}{2}\rb \geq \frac{1}{(\sqrt{2\pi})^{n+1}} \exp \lb -\frac{\|x\|^2}{2 \lambda_n}\rb.
\end{align}
As before, we use the lower bound for $\lambda_n$ from Result \ref{l2} to get a lower bound on the density in \eqref{eq 90} and then use this to get a lower bound for the l.h.s. of \eqref{eq 89}. 


\begin{Rem} Result \ref{l2} gives a lower bound for $\lambda_n$ and the only assumption on $\mu$ required to establish this result is \eqref{a1}. Hence it is the same assumption, namely \eqref{a1},  which gives both  the following bounds:	
	\begin{itemize}[ align=left,leftmargin=*,widest={9}]
		\item an upper bound for the small ball estimate in \eqref{ssbbpp},
		\item a lower bound for the probability of the event $\{X_{t_0}<0, X_{t_1}>0,\ldots\}$ in  \eqref{eq 89}.
		\end{itemize}
	This is not very surprising and we can perceive this as: \textit{the propensity of $X$ to oscillate makes it difficult  for it to be confined to a small ball}. 
	\end{Rem}


	\subsection{Proof idea of Theorem \ref{thmtwodim}} \label{secideatwo} We now present some deterministic results about nodal length of a smooth function which will be useful in proving Theorem \ref{thmtwodim}. 
Let us first introduce some notations. For $f: \re^2 \ra \re$ a smooth function, $T >0$ and $t \in [0,T]$, we define $N_{1,T}$ and $N_{2,T}$ as follows
\begin{align*}
N_{1,T}(t) &:= \#\{y \in [0,T]: f(t,y)=0\}, \\
N_{2,T}(t) &:= \#\{x \in [0,T]: f(x,t)=0\}.
\end{align*}
\noindent The following result in an easy consequence of Lemma \ref{resmain}. 
\begin{Cor} \label{cor1} Let $T, M >0$ and $n \in \nat$  be such that $2T <n$. Suppose $f:[0,n] \ra \re$ is a smooth function such that $\|f^{(n)}\|_{L^{\infty}[0,n]} \leq M$ and $\| f\|_{L^{\infty}[T,2T]} > M(2T)^n/n!$, then $f$ has no more than $(n-1)$ zeros in the interval $[0, T]$. 
\end{Cor}
The following integral geometric result gives a bound for the nodal length of a smooth function and it  appears in the proof of Lemma 5.11 from \cite{DF}. 

\begin{Result} \label{lem11} Let $f:\re^2 \ra \re$ be a smooth function and let $T>0$, then we have the following bound for the nodal length of $f$
	\begin{align*}
	\mbox{length}\{z \in [0,T]^2 : f(z)=0~& \mbox{and}~\nabla f(z)\ne 0\} \nonumber\\
	& \leq \sqrt{2}\left(\int_{0}^{T} N_{1,T}(x) dx + \int_{0}^{T} N_{2,T}(y) dy \right). 
	\end{align*} 
\end{Result}

\noindent The following lemma is an immediate consequence of Corollary \ref{cor1}  and is inspired by Lemma 5.11 in \cite{DF}. 
\begin{Lemma} \label{cor11} Let $T, M>0$, $n \in \mathbb{N}$ be such that $2T \leq n$. Let $f:\re^2 \ra \re$ be a smooth function satisfying the following
	\begin{align*}
	\max\left\{\left\| \partial_2 f \right\|_{L^{\infty}[0,n]^2}, \| \partial_{1}^{n} f\|_{L^{\infty}[0,n]^2}\right\rbrace \leq M/2 \text{ and }\| f(\cdot,0) \|_{L^{\infty}[T,2T]} > M(2T)^n/n!.
	\end{align*}
	Then for every $t \in [0,(2T)^n/n!]$, we have $N_{2,T}(t) < n$.
\end{Lemma}
\begin{proof} Let $x_0 \in [T, 2T]$ be such that $|f(x_0,0)| > M(2T)^n/n!$, then for every $t \in [0,(2T)^n/n!]$ we have
	\begin{equation*}
	\begin{gathered}
	f(x_0, t) - f(x_0, 0) =	\int_{0}^{t} \partial_2 f(x_0,s)~ds,\\
\text{hence }	|f(x_0, t) - f(x_0, 0)| < \frac{M}{2} \cdot \frac{(2T)^n}{n!},\\
\text{hence }	|f(x_0, t)| > \frac{M (2T)^n}{2n!}.
	\end{gathered}
	\end{equation*}	
	\vskip .2cm
	\noindent Thus we have $\| f(\cdot,t) \|_{L^{\infty}[T, 2T]} \geq M (2T)^n/2n!$ and $\| \partial_{1}^{n} f (\cdot, t)\|_{L^{\infty}[0,n]^2} \leq M/2$. Corollary \ref{cor1} now gives the desired result. 
\end{proof}
\noindent The following lemma which is a direct and easy consequence of Lemma \ref{cor11} is the main ingredient  in the proof of Theorem \ref{thmtwodim}. 

\begin{Lemma}\label{newlenlem} Let $T, M>0$, $n \in \mathbb{N}$ be such that $2T \leq n$ and define $\delta := (2T)^n/n!$. Suppose that $f:\re^2 \ra \re$ is a smooth function, then we have the following. 
	\begin{itemize}[align=left,leftmargin=*,widest={10}]
		\item If for every $r \in \{0,1,2,\ldots,\lfloor \delta^{-1}T \rfloor\}$, the following bounds hold for $f$
		\begin{align} \label{con1}
		\max\left\{\left\| \partial_2 f \right\|_{L^{\infty}[0,n]^2}, \| \partial_{1}^{n} f\|_{L^{\infty}[0,n]^2}\right\rbrace \leq M/2 \text{ and }\| f(\cdot,r\delta) \|_{L^{\infty}[T,2T]} > M\delta,
		\end{align}
		then for every $t \in [0,T]$ we have $N_{2,T}(t) < n$.
		\vskip .2cm
		\item If for every $r \in \{0,1,2,\ldots,\lfloor \delta^{-1}T \rfloor\}$, the following bounds hold for $f$
		\begin{align}\label{con2}
		\max\left\{\left\| \partial_1 f \right\|_{L^{\infty}[0,n]^2}, \| \partial_{2}^{n} f\|_{L^{\infty}[0,n]^2}\right\rbrace \leq M/2 \text{ and }\| f(r\delta,\cdot) \|_{L^{\infty}[T,2T]} > M\delta,
		\end{align}
		then for every $t \in [0,T]$ we have $N_{1,T}(t)< n$. 
	\end{itemize}
	Thus if both \eqref{con1} and \eqref{con2} hold for $f$,  we can conclude from Result \ref{lem11} that 
	\begin{align*}
	\mbox{length}\{z \in [0,T]^2 : f(z)=0~ \mbox{and}~\nabla f(z)\ne 0\} 
	\leq 4nT. 
	\end{align*} 
\end{Lemma}
\begin{figure}[ht]
	\def\svgwidth{.5\linewidth}
	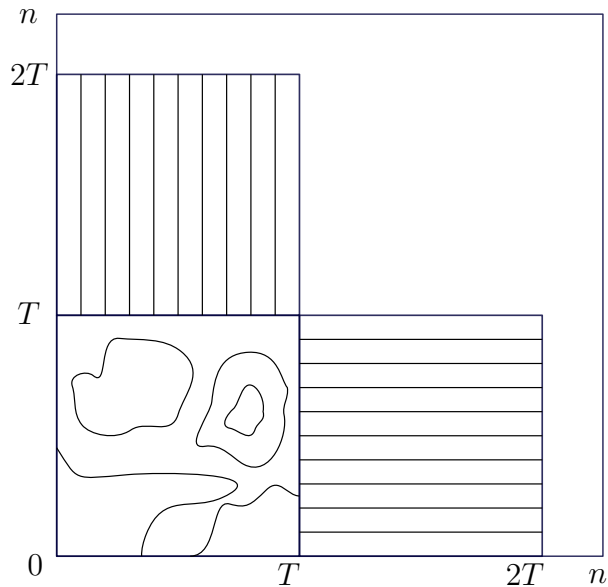
	\caption{An illustration of Lemma \ref{newlenlem}. Derivative bounds for $f$ on $[0,n]^2$ and lower bound for $|f|$ on the $\delta$-separated horizontal and vertical lines in $[T,2T]\times [0,T]$ and $[0,T] \times [T,2T]$ respectively give an upper bound for the nodal length of $f$ in $[0,T]^2$.}
	\label{fig:fig1}
\end{figure}
\noindent The following result is an immediate consequence of Lemma \ref{newlenlem}. 
\begin{Lemma}\label{lem_new} Let $T, M>0$, $n \in \mathbb{N}$ be such that $2T \leq n$ and define $\delta := (2T)^n/n!$. Suppose that $F:\re^2 \ra \re$ is a smooth random function such that almost surely $F$ has no singular zeros, then we have
	\begin{align}\label{eq91}
	\p(\lt \geq 4nT) &\leq \sum_{j=1,2}[\p(\mathscr{E}_j) + \p(\mathscr{F}_j)] + \sum_{r \leq \lfloor \delta^{-1}T \rfloor}[\p(\mathscr{A}_r) +  \p(\mathscr{B}_r)],
	\end{align}
	where the events $\mathscr{E}_j$, $\mathscr{F}_j$, $\mathscr{A}_{r}$ and $\mathscr{B}_{r}$ are defined as follows
	\begin{equation*}
	\begin{gathered}
	\mathscr{E}_j  := \{\| \partial_j F \|_{L^{\infty}[0,n]^2} \geq M/2\},~
	\mathscr{F}_j := \{\| \partial_{j}^{n} F \|_{L^{\infty}[0,n]^2} \geq M/2\},\\
	\mathscr{A}_{r} := \{\| F(\cdot,r\delta) \|_{L^{\infty}[T,2T]} \leq M\delta \},~\mathscr{B}_{r} := \{\| F(r\delta, \cdot) \|_{L^{\infty}[T,2T]} \leq  M \delta \}.
	\end{gathered}
	\end{equation*}
	\end{Lemma}

\subsubsection*{Proof sketch of Theorem \ref{thmtwodim}}  The spectral measure $\mu$ of $X$ satisfies \eqref{a2}, we assume without loss of generality that $v_1 = (1,0)$ and $v_2 = (0,1)$ in \eqref{a2}. Hence both the marginals $\mu_1$ and $\mu_2$ are symmetric probability measures on $\re$ which satisfy \eqref{a1}. 
Let $X_1$ and $X_2$ be centered stationary Gaussian processes on $\re$ whose spectral measures are $\mu_1$ and $\mu_2$ respectively. 

We further assume that $\mu$ is not supported on any line $\re v$, the degenerate case supp$(\mu) \subseteq \re v$ (for some $v \in \re^2$) is much easier and will be analysed  in Section \ref{sectwodim}. In this case, it follows from Bulinskaya's lemma (Result \ref{lembul}) that almost surely $X$ does not have singular zeros. 
Hence we can use Lemma \ref{lem_new} and the stationarity of $X$ to conclude that
\begin{align}\label{eq92}
 \p(\lt \geq 4nT) \leq \sum_{j=1,2}[\p(\mathscr{E}_j) + \p(\mathscr{F}_j)] + \lceil \delta^{-1}T \rceil ~[\p(\mathscr{A}_0) +  \p(\mathscr{B}_0)],
\end{align}
and we now see how to get bounds for the terms on the r.h.s. of \eqref{eq92}. Since $X$ is a stationary Gaussian process, so are all its derivatives. Like in the one dimensional case, estimating $\p(\mathscr{E}_j)$ and $\p(\mathscr{F}_j)$ using 
 the well known tail bounds for supremum of Gaussian processes is quite straightforward and we do this in Section \ref{sectwodim}.

 We now observe that $X(\cdot,0) \overset{d}{=} X_1$ and  $X(0,\cdot) \overset{d}{=} X_2$. Hence $\mathscr{A}_0$ and $\mathscr{B}_0$ correspond to small ball events for $X_1$ and $X_2$ respectively. By our assumption, both their spectral measures satisfy \eqref{a1} and a way to  estimate these small ball probabilities was already discussed in the few lines leading up to \eqref{ssbbpp}. 
 
 Similar to the one dimensional case, for fixed $n,T$ we need to make an optimal choice of $M$ so that the r.h.s. of \eqref{eq92} is minimized.

	\section{Preliminaries}\label{secprelim}

	\noindent {In this section we present some preliminary results and recall other known results which will be used in establishing Theorems \ref{thmonedim} and \ref{thmtwodim}.}

	\subsection{Small Ball Probability} \label{secsmallball} We now get small ball probability estimates for stationary Gaussian processes on $\re$ whose spectral measures satisfy \eqref{a1}. 
	
	\indent Let $X$ be a Gaussian process on an interval $I \st \re$ and let $\{t_j\}_{j=1}^{m} \subset I$ be distinct points in $I$ such that the Gaussian vector  $(X_{t_1},X_{t_2},\ldots,X_{t_m})$ is non-degenerate and let $\phi$ be its density. Then for every $\eta>0$, we have 
	\begin{align}
	\p(\| X \|_{L^{\infty}(I)} \leq \eta) & \leq \p(|X_{t_1}|\leq \eta,\ldots,|X_{t_m}|\leq \eta), \nonumber\\ 
	& = \int_{[-\eta,\eta]^m} \phi (x) dx \leq  2^m  \|\phi\|_{L^{\infty}(\re^m)}~ \eta^m. \label{smallball}
	\end{align}
	For a stationary Gaussian process whose spectral measure satisfies \eqref{a1}, we can get an estimate for  $\|\phi\|_{L^{\infty}(\re^m)}$ which appears in \eqref{smallball} and this is what is done below. 
	\begin{Result}[\cite{FN}, Turan's lemma] \label{Turan} Let $p(t) = \sum_{k=1}^{n} c_k e^{i\lambda_k t}$, where $c_k \in \mathbb{C}$ and $\lambda_k \in \mathbb{R}$. Then there is a constant $ A>0$ such that for every interval $I \subseteq \re$, every measurable set $E \st I$ and every $q \in [0,\infty]$, we have
		\begin{align*}
		\| p\|_{L^{q}(I)} \leq \lb \frac{A|I|}{|E|} \rb ^{n-1} \| p\|_{L^{q}(E)},
		\end{align*}
		where $|\cdot|$ denotes the Lebesgue measure. 
	\end{Result}
	
	\begin{Fact} \label{f1} Let $X$ be a centered stationary Gaussian process on $\re$ whose spectral measure $\mu$ satisfies \eqref{a1}. Hence $\mu$ admits the following decomposition $d\mu(x) = f(x)dx + d\mu_s (x)$, where $f \not \equiv 0$. Then there is  a constant  $\delta_0 >0$ such that $|\{f \geq \delta_0\}| \geq \delta_0$ and hence there is a large enough $M_0 > \pi$  such that $|S|\geq \delta_0/2$, where
		$S := \{f \geq \delta_0\} \cap (-M_0,M_0)$.
		For $m \in \nat$ and $T>0$ consider the stationary Gaussian process $Y$ on $\mathbb{Z}$ defined by
		\begin{align}\label{ydef}
		Y_{\ell} := X_{\ell T/m},~\text{for }\ell \in \mathbb{Z}.
		\end{align}
		Then the spectral measure $\mu_{m,\deltat}$ of $Y$ is  a symmetric probability measure on $[-\pi,\pi]$ which is the push forward of $\mu$ by the map $\psi$ given by 
		\begin{align*}
		\psi: \re \longrightarrow &~ \mathbb{S}^1 \simeq [-\pi,\pi)\\
		 x \longmapsto &~  \frac{Tx}{m} ~(\text{mod}~ 2\pi).
		\end{align*}
	Hence  $\mu_{m,T}$ considered as a measure on $[-\pi,\pi]$ has a nontrivial absolutely continuous part w.r.t. the Lebesgue measure on $[-\pi,\pi]$ given by $f_{m,T}(x) dx$, where 
	\begin{align*}
	f_{m,T} (x) := \sum_{n \in \mathbb{Z}}  \frac{m}{T}~ f\lb \frac{m}{T}(x+ 2\pi n)\rb \geq \frac{m}{T}~ f\lb \frac{mx}{T} \rb,\text{ for $x\in [-\pi,\pi]$}.
	\end{align*}
Define $b := \pi/M_0$, if $\deltat$ and $m$ are such that  $\deltat \leq b m$, then $(T/m)S \subset [-\pi,\pi]$ and hence 	
	\begin{align}\label{sineq}
	|\{f_{m,T} \geq m\delta_0/T\}| \geq |(T/m)S| \geq T\delta_0/2m. 
	\end{align}
	\end{Fact}
	
	The following result taken from the proofs of Lemma 3 and Theorem 2 of \cite{MK} is the main tool in obtaining the small ball probability estimates. 
	
	\begin{Result}\label{l2} Let $X$ be a centered stationary Gaussian process on $\re$ whose spectral measure $\mu$ satisfies \eqref{a1}. Let $T>0$, $m \in \nat$ and $Y$ be the  Gaussian process on $\mathbb{Z}$ defined in \eqref{ydef}. Let $\Sigma$ and $\phi$ denote the covariance matrix and density of the Gaussian vector $ (Y_1, \ldots, Y_m)$ respectively. Let $\lambda$ be the smallest eigenvalue of $\Sigma$. Then there exists $b, c \in (0,1)$ and $C>1$ such that whenever $T \leq bm$, we have
		\begin{align*}
		\lambda \geq \lb\frac{cT}{m} \rb^{2(m-1)}\text{ and }~ \|\phi\|_{L^{\infty}(\re^m)} \leq  \left(\frac{Cm}{\deltat}\right)^{m^2}.
		\end{align*}
		\end{Result}
	\begin{proof}  Let $f$, $M_0$, $b$, $\delta_0$ and ${S}$ be as in Fact \ref{f1}. The density $\phi$ is given by
		\begin{align}
		\phi (x) & = \frac{1}{(2\pi)^{m/2} |\Sigma|^{1/2}} \exp\left(-\frac{\langle \Sigma^{-1}x,x \rangle}{2}\right), \nonumber \\
		 \text{hence } \|\phi\|_{L^{\infty}(\re^m)}  & \leq \frac{1}{(2\pi)^{m/2} |\Sigma|^{1/2}} \leq \frac{1}{(2\pi)^{m/2}} \frac{1}{\lambda^{m/2}}. \label{eq103}
		\end{align}
		 We now get a lower bound on $\lambda$ and use this to get an upper bound on $\phi$. Let $u =(u_1,u_2,\ldots,u_m) \in \re^m$ and define $U(x) := \sum_{k=1}^{m}u_k e^{ikx}$. Then we have
		\begin{align*}
		\int_{-\pi}^{\pi} |U(x)|^2 dx  = 2\pi \|u\|^2 \text{ and }
		\int_{-\pi}^{\pi} |U(x)|^2 d\mu_{m,T}(x)  = \langle \Sigma u,u \rangle .
		\end{align*}
		Hence it follows from \eqref{sineq} that 
		\begin{equation}\label{eq100}
		\begin{aligned}
		\langle \Sigma u,u \rangle & = \int_{-\pi}^{\pi} |U(x)|^2 d\mu_{m,T}(x)
		 \geq \int_{-\pi}^{\pi} f_{m,T}(x) |U(x)|^2 dx,\\
		 &  \geq \frac{m}{\deltat}\delta_0 \int_{(T/m)S}  |U(t)|^2 dt. 
		\end{aligned}
		\end{equation}
		We now use Turan's lemma (Result \ref{Turan}) to get 
		\begin{align}
		\int_{(T/m)S} |U(t)|^2 dt \geq \left \lbrace \frac{2\pi A}{|(T/m)S|} \right\rbrace ^{-2(m-1)} \int_{-\pi}^{\pi} |U(t)|^2 dt, \label{eq101}
		\end{align}
		and hence we conclude the following from \eqref{sineq}, \eqref{eq100} and \eqref{eq101}   
		\begin{align}
		\langle \Sigma u, u \rangle \geq \frac{m\delta_0}{\deltat} \left(\frac{\deltat\delta_0}{4\pi mA} \right)^{2(m-1)} 2\pi \|u\|^2, \nonumber \\
		\text{hence, } \lambda \geq 2\pi \cdot \frac{m\delta_0}{\deltat} \left(\frac{\deltat\delta_0}{4\pi mA} \right)^{2(m-1)} \geq \lb \frac{cT}{m}\rb^{2(m-1)}, \label{eq102}
		\end{align}
		for some $c \in (0,1)$. Using the lower bound  we get for $\lambda$ from \eqref{eq102} in \eqref{eq103}, we get the desired result 
		\begin{align*}
		 \|\phi\|_{L^{\infty}(\re^m)} & \leq \left(\frac{Cm}{\deltat}\right)^{m^2}.
		\end{align*}
	\end{proof}
\noindent	The following small ball probability estimate is a consequence of \eqref{smallball} and Result \ref{l2}. 
	\begin{Lemma}[Small ball probability] \label{sbp} Let $X$ be a centered stationary Gaussian process on $\re$ whose spectral measure satisfies \eqref{a1}. Then there are constants $b \in (0,1)$ and $C>1$ (both depending only on $X$) such that for every $\eta >0$ and every {$\deltat >0$, $m \in \nat$ satisfying $\deltat \leq b m$}, we have
		\begin{align}\label{smallballpp}
		\p(\| X \|_{L^{\infty}[0,\deltat]} \leq \eta) \leq \left(\frac{Cm}{\deltat}\right)^{m^2}\eta^m.
		\end{align}
	\end{Lemma}

\noindent In the following Lemma, we get small ball probability estimates by using Lemma \ref{sbp} with specific values of $\eta$, $m$ and $T$. This will be used in the proofs of Theorems \ref{thmonedim} and  \ref{thmtwodim}.
\begin{Lemma}\label{lemsbp} Let $X$, $b$, $C$, $T$ and $m$ be as in Lemma \ref{sbp}. Let $\ep \in (0,1/2)$, $n \in \nat$ be such that $m =\lfloor \ep n\rfloor$, $n \geq 1/\ep^2$. Let $A >1$, $B = (4eAC)^{-1}$, $M= 2A\mathscr{D}_n \mathscr{H}_n$ where $\mathscr{D}_n \geq 1$ and $\mathscr{H}_n$ is defined by
	\begin{align*}
	\mathscr{H}_n := \sqrt{n^2 h_n},\text{ where }h_n := \log \left( \frac{B}{\mathscr{D}_{n}^{1/n}} \cdot \lb \frac{n}{\deltat}\rb^{1-2\ep} \right).
	\end{align*}
	Suppose $h_n \geq 1$, then with $\eta = M(2T)^n/n!$ the estimate in Lemma \ref{sbp} becomes 
	\begin{align*}
		\p(\| X \|_{L^{\infty}[0,\deltat]} \leq M(2T)^n/n!) \leq \exp \left\{ \frac{-\ep n^2}{2} \log \lb \frac{B}{\mathscr{D}_{n}^{1/n}} \cdot \lb \frac{n}{T}\rb^{1-2\ep} \rb \right\}.
	\end{align*}
\end{Lemma}

\begin{proof}
	Let $W$ denote the expression in the r.h.s. of \eqref{smallballpp} for $\eta = M(2T)^n/n!$, then  we have
	\begin{align}
	W& = \left( \frac{Cm}{\deltat}\right)^{m^2} \lb \frac{M}{n!} (2\deltat)^n \rb^m,\nonumber\\
	& \leq \frac{C^{mn} n^{m^2}}{\deltat^{m^2}} \cdot \frac{(2e)^{mn}~ \deltat^{mn}~ (2A\mathscr{H}_n \mathscr{D}_n)^m}{n^{nm}}, \nonumber\\
	& \leq  \lb\frac{n}{\deltat}\rb^{m^2 - nm} \cdot (B^{-1}  \mathscr{D}_{n}^{1/n})^{mn} \cdot  \mathscr{H}_{n}^{m}.\label{lline}
	\end{align}
	Since $m = \lfloor \ep n\rfloor$ and $n \geq 1/\ep^2$, we have
	\begin{align*}
	m^2 - mn \leq\ep^2 n^2 -  (\ep n-1)n  = - \ep n^2\lb 1- \ep -(1/\ep n) \rb \leq -\ep (1-2\ep) n^2.
	\end{align*}
	Also because $n/T$, $B^{-1} \mathscr{D}_{n}^{1/n}$ and $\mathscr{H}_n$ are all greater than $1$, we have the following upper bound for the expression on the r.h.s. of \eqref{lline}
	\begin{align}
	W & \leq \exp(-\ep (1-2\ep)n^2 \log(n/T) + \ep n^2 \log (B^{-1} \mathscr{D}_{n}^{1/n})+ \ep n \log \mathscr{H}_n) ) \nonumber, \\
	& \leq  \exp \left\{-\ep n^2 \log \left( \frac{B}{\mathscr{D}_{n}^{1/n}} \cdot \lb \frac{n}{\deltat}\rb^{1-2\ep} \right) + \ep n \log \mathscr{H}_n\right\}, \nonumber\\
	& = \exp(-\ep n^2 h_n + \ep n \log \sqrt{n^2 h_n}),\nonumber\\
	& \leq \exp(-\ep n (n h_n -  \log (n h_n)) \leq \exp(-\ep n^2 h_n/2), \label{eqv3}
	\end{align}
	where the inequality in \eqref{eqv3} follows from our assumption that $h_n \geq 1$ and this establishes our claim. 
	\end{proof}

	\subsection{Metric Entropy and Supremum of Gaussian Processes} \label{secsupGP}
	We now recall well known results which give tail bounds for the supremum of Gaussian processes and we shall use them to get tail bounds for higher derivatives of stationary Gaussian processes.  The definitions and results in this section are taken from Section 3.4 of \cite{PM}. 
	
	Let $\mathscr{T}$ be an index set and  let $(X_t)_{t \in \T}$ be a centered Gaussian process on $\T$. Then $X$ induces a  pseudo-metric $d$ on $\T$ given by
	\begin{equation*}
	d(s,t) := (\e [X_s - X_t]^2)^{1/2}.
	\end{equation*}
	Suppose $(\T,d)$ is totally bounded and let $\ep >0$. The \textit{$\epsilon$-covering number} denoted by $N'(\epsilon,\T)$ is defined to be the minimal number of $\epsilon$-balls required to cover $(\T,d)$. The \textit{$\epsilon$-packing number} denoted by $N(\epsilon, \T)$ is defined to be the maximal possible cardinality of subset $A \subset \T$ which is such that $d(s,t) > \epsilon$, for distinct $s,t \in A$. The quantities $N$ and $N'$ are related by
	\begin{equation}
	N'(\epsilon,\T) \leq N(\epsilon, \T) \leq N' \left(\ep/2,\T\right). \label{eq1}
	\end{equation}
	The \textit{$\epsilon$-entropy number} $H(\epsilon,\T)$ is defined as $H(\epsilon,\T) :=\log N(\epsilon,\T)$. The following results are stated in terms of $N$, but because of \eqref{eq1} similar results also hold with $N$ replaced by $N'$. 
	
	\begin{Result}[\cite{PM}, Theorem 3.18] \label{thment} Let $(X_t)_{t \in \T}$ be a centered Gaussian process and assume that $(\T,d)$ is totally bounded. If $\sqrt{H(\cdot,\T)}$ is integrable at $0$, then $X_t$ admits a version which is almost surely uniformly continuous on $(\T,d)$ and for such a version we have 
		\begin{equation*} 
		\e [\sup_{t \in \T} X_t] \leq 12 \int_{0}^{\sigma} \sqrt{H(x,\T)}~ dx,
		\end{equation*}
		where $\sigma = (\sup_{t \in \T} \e[X_{t}^{2}])^{1/2}$.
		
	\end{Result}
	
	\begin{Fact} For $(X_t)_{t \in \T}$  a centered Gaussian process on an index set $\T$, we have 
		\begin{equation}
		\e[\sup_{t \in \T} |X_t|] \leq \e[|X_0|] + 2 \e[\sup_{t \in \T} X_t]. \label{fact}
		\end{equation}
	\end{Fact}
	\noindent The following result gives tail bounds for the supremum of a Gaussian process.
	\begin{Result}[\cite{PM}, Proposition 3.19]\label{expconc} Let $(X_t)_{t \in \T}$ be an almost surely continuous centered Gaussian process on the totally bounded set $(\T,d)$. Let $Z$ denote either $\sup_{t \in \T} X_t$ or $\sup_{t \in \T} |X_t|$ and let $\sigma = (\sup_{t \in \T} \e[X_{t}^{2}])^{1/2}$. Then for every $x>0$, we have
		\begin{align*}
		\p(Z - \e[Z] \geq \sigma \sqrt{2x}) \leq \exp(-x),\\
		\p(\e[Z] - Z \geq \sigma \sqrt{2x}) \leq \exp(-x).
		\end{align*}
	\end{Result}
\subsection{Derivatives of a stationary Gaussian process}  We now present some basic facts about  stationary Gaussian processes on $\re$/$\re^2$. Let $X$ be a centered stationary Gaussian process on $\re$/$\re^2$ with spectral measure $\mu$ and covariance function $\e[X_t X_0] = k(t) = \hat{\mu}(t)$. Further assume that $\mu$ is a probability measure all of whose moments $C_n$/$C_{m,n}$ are finite. Then we have the following.
\begin{enumerate}[label={\arabic*.},align=left,leftmargin=*,widest={7}]
	\item For $X$ on $\re$ and $t \in \re$, we have
	\begin{itemize}[align=left,leftmargin=*,widest={7}]
		\item $X$ is smooth and for every $n \in \nat$, $X^{(n)}$ is a centered stationary Gaussian process with  covariance function given by 
		\begin{align} \label{covkerder}
		\e[X^{(n)}_t X^{(n)}_0] = (-1)^n k^{(2n)}(t).
		\end{align}
		\item The derivatives of $k$ are given by
		\begin{align}
		k^{(n)}(t) & = \int_{\re} (ix)^n e^{-itx} d\mu(x), \nonumber \\
		\|k^{(n)}\|_{L^{\infty}(\re)} & \leq \int_{\re} |x|^n  d\mu(x) = C_n. \label{ineq}
		\end{align}
		\end{itemize}
	\item For $X$ on $\re^2$ and  $z=(x,y) \in \re^2$ and $s \in \re^2$, we have
	\begin{itemize}[align=left,leftmargin=*,widest={7}]
		\item $X$ is smooth and for every $m,n \in \nat \cup \{0\}$, $\partial_{1}^{m} \partial_{2}^{n}X$ is a centered stationary Gaussian process with covariance function given by
		\begin{align*}
		\e[\partial_{1}^{m} \partial_{2}^{n}X(z)\cdot  \partial_{1}^{m} \partial_{2}^{n}X(0)] = (-1)^{m+n} \partial_{1}^{2m} \partial_{1}^{2n}k (z).
		\end{align*}
		\item The derivatives of $k$ are given by
			\begin{align*}
		\partial_{1}^m \partial_{2}^{n} k(s) & = \int_{\re^2} (ix)^m (iy)^n e^{i\langle s,z\rangle} d\mu(z),\\
		\Vert \partial_{1}^m \partial_{2}^{n} k\Vert_{L^{\infty}(\re^2)} & \leq \int_{\re^2} |x|^m |y|^n~ d\mu(z)= C_{m,n}.
		\end{align*}
		\end{itemize}
	\end{enumerate} 
\vskip .2cm
Bulinskaya's lemma (\cite{NS}, Lemma 6) for a stationary Gaussian process reads as follows. 
\begin{Result}\label{lembul} Let $X$ be a stationary Gaussian process on $\re^2$ with spectral measure $\mu$ which is a symmetric probability measure on $\re^2$. Then almost surely $X$ does not have any singular zero (a singular zero is a  point $z \in \re^2$ such that $X(z) = |\nabla X(z)| = 0$) if the following conditions holds.
	\begin{itemize}[align=left,leftmargin=*,widest={7}]
		\item $\max\{C_{4,0}, C_{0,4}\} < \infty$.
		\item The Gaussian vector $\nabla X(0) = (\partial_{1} X(0), \partial_{2} X(0))$ is non-degenerate, which holds iff there is no $v \in \re^2$ for which supp$(\mu) \subseteq \re v$. 
		\end{itemize}
		\end{Result}
 \begin{Rem}\label{remasymp} 
	We now see how the moments $L_n$ and $R_n$ defined in \eqref{defns2} are related. We first note that
	\begin{align*}
	C_{m,n} \leq \int_{\re^2} (x^2+y^2)^{\frac{m+n}{2}} d\mu(x,y) = \int_{0}^{\infty} t^{m+n} d\murad(t),
	\end{align*}
	and hence $\widetilde{R}_{n} \leq \widetilde{L}_n$. We also have 
	\begin{align*}
	\widetilde{L}_{n}^{2} &= \int_{\re^2} (x^2 + y^2)^n d\mu(x,y)  \leq \int_{\re^2} 2^n \lb \frac{x^{2n} + y^{2n}}{2}\rb d\mu(x,y) \leq 2^n \widetilde{R}_{n}^{2}
	\end{align*}
	and hence we conclude that 
	\begin{align} \label{momrel}
	R_{n}^{1/n} \leq L_{n}^{1/n} \leq \sqrt{2} R_{n}^{1/n}.
	\end{align}
\end{Rem}

\section{Proof of Theorem \ref{thmonedim}}\label{seconedim}
\subsubsection*{Upper bound} As discussed in Section \ref{secideaone}, we use Lemma \ref{lem2_onedim} to prove the upper bound in Theorem \ref{thmonedim} and hence we only need to estimate the two terms on the r.h.s. of \eqref{main_onedim}. 
	Let $b$, $C$ be as in Lemma \ref{sbp} and define $B := (4eAC)^{-1}$, where $A = 192 \sqrt{\pi}$. Let $\ep \in (0,1/2)$, $n,m \in \nat$ and $T>0$ be such that $n \geq 1/\ep^2,~ m =\lfloor \ep n\rfloor \text{ and } {T\leq b \lfloor  \ep  n \rfloor}$. We also assume that $h_n \geq 1$, where $h_n$ is defined by
	\begin{align*}
	h_n :=  \log \left( \frac{B}{D_{n}^{1/n}} \cdot \lb \frac{n}{\deltat}\rb^{1-2\ep} \right).
	\end{align*}
	With $D_n$ as in \eqref{defns}, we choose $M = 2A H_n D_n$, where $H_n := \sqrt{n^2 h_n}$. 
	
	For this  choice of $M$, we get the following estimate for the  small ball probability in \eqref{main_onedim} from Lemma \ref{lemsbp} 
	\begin{align}
	\p \lb \lVert X \rVert_{L^{\infty}[\deltat, 2\deltat]} \leq \frac{M}{n!}(2\deltat)^n \rb \leq \exp \left\{ \frac{-\ep n^2}{2} \log \lb \frac{B}{{D}_{n}^{1/n}} \cdot \lb \frac{n}{T}\rb^{1-2\ep} \rb \right\}. \label{smallballp}
	\end{align} 

			 We now estimate $\p(\lVert X^{(n)} \rVert_{L^{\infty}[0, 2\deltat]} > M)$, which is the other term in \eqref{main_onedim}, using Results \ref{thment} and \ref{expconc}. 	Let $k := \widehat{\mu}$ be the covariance function of $X$. The following calculation is to estimate the entropy numbers for the process $(X^{(n)}_t)_{t \in [0,2\deltat]}$. For this we first get an estimate of the $\epsilon$-covering number. 
			We denote  by $d_n$ the pseudo-metric on $\re$ induced by $X^{(n)}$. 
			\begin{align}
			d_n(t,0)^2 &= \e[(X^{(n)}_t -X^{(n)}_0)^2] = 2(-1)^n [k^{(2n)}(0)-k^{(2n)}(t)], \nonumber \\
			 & = 2(-1)^{n+1} \int_{0}^{t} [k^{(2n+1)}(s) - k^{(2n+1)}(0)] ds, \nonumber \\
			& \leq 2 \int_{0}^{t} \Vert k^{(2n+2)} \Vert_{L^{\infty}(\re)}~ s ds \leq C_{2n+2}~ t^2, \label{eq300}
			\end{align}
			where the last inequality follows from mean value theorem and \eqref{ineq}. Let $\mathcal{B}_n(x,r)$ denote the ball in $(\re, d_n)$ with center $x$ and radius $r$. Then it follows from \eqref{eq300} that 
			\begin{align*}
		(x-\ep,x+\ep)	\subseteq  \mathcal{B}_n \lb x,\sqrt{C_{2n+2}}~ \ep \rb,
			\end{align*}
		and hence we get the following bound for the  $\epsilon$-covering number $N'(\ep)$ of $([0,2\deltat],d_n)$
			\begin{align}
		N'(\epsilon) & \leq \frac{\sqrt{C_{2n+2}}}{\epsilon}~ 2\deltat. \label{eq6}
		\end{align}	
			We now get a bound on the $\epsilon$-packing number $N(\ep)$ using \eqref{eq1} and \eqref{eq6}
			\begin{align*}
			N(\epsilon) &\leq N' \left(\frac{\epsilon}{2}\right) \leq \frac{4\sqrt{C_{2n+2}}}{\epsilon}~ \deltat = \frac{\beta}{\epsilon},~\mbox{where $\beta =4\sqrt{C_{2n+2}}~\deltat$}.
			\end{align*}
			The above inequality holds for $\epsilon < \beta$ and for $\epsilon \geq \beta$, we have $N(\epsilon)=1$. Hence we have
			$$ H(\epsilon) 
			\begin{cases}
			= 0, & \text{if }\epsilon \geq \beta, \\
			\leq \log \left(\beta /\epsilon\right), & \text{if }\epsilon < \beta.
			\end{cases}
			$$
			\begin{align}\label{dudint}
			 \int_{0} ^{\infty} \sqrt{H(x)} dx = \int_{0} ^{\beta} \sqrt{H(x)} dx &\leq \int_{0} ^{\beta} \sqrt{\log \left(\frac{\beta}{x}  \right)} dx = 2\beta \int_{0}^{\infty} y^2 e^{-y^2} dy  = \sqrt{\pi} \beta.
			\end{align}
	We now conclude from Result \ref{thment} that
		\begin{equation}
		\e [\sup_{t \in [0,2\deltat]} X^{(n)}_t] \leq 12\sqrt{\pi} \beta  = 48\sqrt{\pi} \sqrt{C_{2n+2}}~\deltat.  \label{eq7}
		\end{equation}
	It now follows from \eqref{fact} and \eqref{eq7} that with $A = 192 \sqrt{\pi}$, we have
	\begin{equation}
		\begin{aligned}\label{expcal}
		\e[\sup_{t \in [0,2\deltat]} |X^{(n)}_t|] &\leq \e[|X^{(n)}_0|] + 96\sqrt{\pi} \sqrt{C_{2n+2}}~\deltat,   \\
		& \leq_{\eqref{covkerder}}  \sqrt{|k^{(2n)}(0)|} + 96\sqrt{\pi} \sqrt{C_{2n+2}}~\deltat,  \\
		& \leq 96\sqrt{\pi}~(\sqrt{C_{2n}} + \sqrt{C_{2n+2}}~\deltat),   \\
		& \leq A \max\{\sqrt{C_{2n}}, \sqrt{C_{2n+2}}~\deltat\} \leq A nD_n. 
		\end{aligned} 
		\end{equation}
	It follows from  Result \ref{expconc} that for every $n \in \nat$ and every $x >0$, we have  
		\begin{equation*}
		\p(\sup_{t \in [0,2\deltat]} |X^{(n)}_t| > A nD_n +\sqrt{C_{2n}}~x) \leq  \exp(-x^2/2), 
		\end{equation*}
	\noindent	 and since $M = 2AH_nD_n \geq AnD_n + D_n H_n \geq AnD_n + \sqrt{C_{2n}} H_n$, we conclude that 
		\begin{align}\label{eqm}
		\p(\sup_{t \in [0,2\deltat]} |X^{(n)}_t| \geq M) \leq \exp(-H_{n}^{2}/2) = \exp(-n^2 h_n/2). 
		\end{align}
Hence we conclude from \eqref{main_onedim}, \eqref{smallballp} and \eqref{eqm} that 
		\begin{align*}
		\p(\nt \geq n) \leq 2\exp \left\{-\frac{\ep n^2}{2} \log \left( \frac{B}{D_{n}^{1/n}} \cdot \lb \frac{n}{\deltat}\rb^{1-2\ep} \right) \right\}. 
		\end{align*}

	\subsubsection*{Lower bound}
	\noindent Let $b>0$ be as in Result \ref{l2}.  For $T >0$ and $n\in \nat$ such that $T \leq bn$, we  define the $(n+1)$ dimensional  Gaussian vector $V_n$  by
	\begin{align*}
	V_n := (Y_0, Y_1, Y_2,\ldots,Y_n),
	\end{align*}
	where $Y_k := X_{kT/n}$.  Let $\Sigma$ denote the covariance matrix of $V_n$. The density $\phi_n$ of $V_n$ is 
	\begin{align*}
	\phi_n(x) = \frac{1}{(\sqrt{2\pi})^{n+1}~ |\Sigma|^{1/2}} \exp \lb -\frac{\langle \Sigma^{-1}x,x \rangle}{2}\rb.
	\end{align*}
	Let $\lambda$ be the smallest eigenvalue of $\Sigma$, then we get the following lower bound for $\lambda$ from  Result \ref{l2}. There is $c >0$ such that 
	$\lambda \geq  (c \deltat/n)^{2n}$, hence we have
	\begin{align*}
	\langle \Sigma^{-1}x,x \rangle \leq \frac{1}{\lambda} \|x\|^2 \leq \lb \frac{n}{c \deltat}\rb^{2n} \|x\|^2.
	\end{align*}
	We have from Lemma 3 in \cite{MK} that $|\Sigma| \leq 1$ and hence we have the following lower bound for the density $\phi_n$
	\begin{align*}
	\phi_n (x) \geq \frac{1}{(\sqrt{2\pi})^{n+1}} \exp \lb  -\lb \frac{n}{c \deltat}\rb^{2n} \frac{\|x\|^2}{2} \rb. 
	\end{align*} 
	It follows from \eqref{eq 89} that 
	\begin{align*}
	\p(\nt \geq n) &\geq \frac{1}{(\sqrt{2\pi})^{n+1}}  \int \cdots \int_{-\infty}^{0} \int_{0}^{\infty}  \exp \lb  -\lb \frac{n}{c \deltat}\rb^{2n} \frac{\|x\|^2}{2} \rb dx,\\
	& \geq \frac{1}{(\sqrt{2\pi})^{n+1}} \cdot \lb \frac{\sqrt{2\pi}}{2} \lb \frac{c \deltat}{n}\rb^{n} \rb^{n+1} \geq  \lb \frac{\deltat}{c' n} \rb^{n^2},\\
	& \geq \exp \lb -n^2 \log \lb c'n/T\rb \rb.
	\end{align*}


		\section{Proof of Theorem \ref{thmtwodim}}\label{sectwodim} We continue to work with the {assumptions and the notations of} Section \ref{secideatwo}. The starting point of our analysis here is \eqref{eq92}, which is 
		\begin{align*}
		\p(\lt \geq 4nT) \leq \sum_{j=1,2}[\p(\mathscr{E}_j) + \p(\mathscr{F}_j)] + \lceil \delta^{-1}T \rceil ~[\p(\mathscr{A}_0) +  \p(\mathscr{B}_0)],
		\end{align*} 
		By our assumption, both $X_1$, $X_2$ satisfy the assumptions of Lemma \ref{sbp} and  let $b \in (0,1)$, $C>1$ be such that the conclusion of Lemma \ref{sbp} holds for both $X_1$, $X_2$ with these constants.   
	Define ${B} := (4eAC)^{-1}$, where $A = 384 \sqrt{\pi}$. Let $\ep \in (0,1/4)$, $n,m \in \nat$ and $T>0$ be such that $n \geq 1/\ep^2,~ m =\lfloor \ep n\rfloor, \text{ } {T\leq b \lfloor  \ep  n \rfloor}$ and $\delta := (2T)^n/n!$.  We further assume that $(B/L_{n}^{1/n}) \cdot (n/T)^{1-4\ep} > e$ and since we have from \eqref{momrel} that $L_n \geq R_n$, it follows that $g_n \geq 1$, where $g_n$ is defined by
			\begin{align*}
			g_n :=  \log \left( \frac{{B}}{R_{n}^{1/n}} \cdot \lb \frac{n}{\deltat}\rb^{1-2\ep} \right).
			\end{align*}
		
			We choose $M=2AR_n G_n$, where $G_n := \sqrt{n^2 g_n}$ and estimate the terms on the r.h.s. of \eqref{eq92}. We first estimate the small ball probabilities $\p(\mathscr{A}_0),~\p(\mathscr{B}_0)$ using Lemma \ref{lemsbp} with $\mathscr{D}_n = R_n$ and $\mathscr{H}_n = G_n$
				\begin{align*}
			\p(\mathscr{A}_0),~\p(\mathscr{B}_0) \leq \exp \left\{ \frac{-\ep n^2}{2} \log \lb \frac{B}{R_{n}^{1/n}} \cdot \lb \frac{n}{T}\rb^{1-2\ep} \rb \right\},
			\end{align*}
			and since $\delta^{-1}T = (n! T)/(2T)^n \leq (n/T)^n \cdot (T/2^n) \ll (n/T)^n$, we have
			\begin{align}\label{eq001}
			\lceil	\delta^{-1}T \rceil~ [\p(\mathscr{A}_0) + \p(\mathscr{B}_0)] & \leq  2 \exp \left\{ \frac{-\ep n^2}{2} \log \lb \frac{B}{R_{n}^{1/n}} \cdot \lb \frac{n}{T}\rb^{1-2\ep - \frac{2}{\ep n}} \rb \right\}, \nonumber\\
			& \leq 2 \exp \left\{ \frac{-\ep n^2}{2} \log \lb \frac{B}{R_{n}^{1/n}} \cdot \lb \frac{n}{T}\rb^{1-4\ep} \rb \right\},
			\end{align}
			where the last inequality follows from our assumption that $n \geq 1/\ep^2$.
			
	Like in the one dimensional case, Results \ref{thment} and \ref{expconc} will be used to estimate the probability of the events $\mathscr{E}_j$ and $\mathscr{F}_j$. We now calculate the expected supremum of the process $\partial_{1}^n X$ on $[0,n]^2$. Let $d_n$  be the pseudo-metric on $\re^2$ induced by the Gaussian process $\partial_{1}^n X$, then for $z \in \re^2$ we have
					\begin{align*}
					d_n(z,0)^2 &= \e\left[\partial_{1}^n X(z)- \partial_{1}^n X(0) \right]^2 = 2(-1)^n [\partial_{1}^{2n} k(0) - \partial_{1}^{2n} k(z)]. 
					\end{align*}
					For $s=(s_1,s_2) \in \re^2$, we have
					\begin{align*}
					|\partial_{1}^{2n} k(0) - \partial_{1}^{2n} k(s)| & = |\int_{0}^{1} \nabla\partial_{1}^{2n} k(ts) \cdot s~ dt |, \\ 
					& = | s_1 \int_{0}^{1} [\partial_{1}^{2n+1}k(ts) - \partial_{1}^{2n+1}k(0)]~dt  \\
					& \hspace{1.2cm} + s_2 \int_{0}^{1} [\partial_{1}^{2n} \partial_{2}k(ts) - \partial_{1}^{2n} \partial_{2}k(0)]dt |,\\
					& \leq  |s_1| \int_{0}^{1} (|s_1| C_{2n+2,0} + |s_2| C_{2n+1,1}) ~tdt  \\
					& \hspace{1.2cm} + |s_2| \int_{0}^{1} (|s_1|C_{2n+1,1} + |s_2| C_{2n,2})~tdt,\\
					& \leq  \max\{ C_{2n+2,0}, C_{2n,2}, C_{2n+1,1}\} \|s\|^2. 
					\end{align*}
					Hence we have $d_n(s,0) \leq  \widetilde{R}_{n+1}\Vert s \Vert \leq {R}_{n}\Vert s \Vert$, where $\widetilde{R}_{n}$ and $R_n$ are defined in \eqref{defns2}. We let $D(x,\ell)$ and $\mathcal{B}_{n}(x,\ell)$ denote the balls centered at $x$ with radius $\ell$ in the Euclidean metric and $d_n$ respectively, then $D(x,\ep/R_n) \subseteq \mathcal{B}_n(x,\ep)$ and hence $N'(\ep)$, the $\ep$-covering number of $([0,n]^2,d_n)$, has the following upper bound $N'(\ep) \leq  {4R_{n}^{2}}n^2/\ep^2$ and hence the $\ep$-packing number $N(\ep)$ satisfies $N(\ep) \leq {8R_{n}^{2}}n^2/\ep^2$. Similar to \eqref{dudint}, we can get the following estimate for the Dudley integral in this case also
					\begin{align*}
					\int_{0}^{\infty} \sqrt{\log N(x)}~ dx \leq 4\sqrt{\pi}n R_n. 
					\end{align*}
					We thus conclude from Result \ref{thment} that
					\begin{equation*}
					\e [\sup_{ [0,n]^2} \partial_{1}^{n} X] \leq  48\sqrt{\pi}nR_n, 
					\end{equation*}
and by a calculation similar to \eqref{expcal}, we conclude that with $A= 384\sqrt{\pi}$ we have
\begin{align*}
\e [~\|\partial_{1}^{n} X\|_{L^{\infty}[0,n]^2}] \leq  AnR_n/2,
\end{align*}
			thus we get the following tail bound from Result \ref{expconc} and the fact that $R_n \geq \sqrt{C_{2n,0}}$
			\begin{align*}	
			\p(\|\partial_{1}^{n} X\|_{L^{\infty}[0,n]^2} > (A nR_n/2) + x R_n) \leq  \exp(-x^2/2). 
			\end{align*}
			For $j=1,2$ and $Y = \|\partial_{j} X\|_{L^{\infty}[0,n]^2}$, $\|\partial_{j}^{n} X\|_{L^{\infty}[0,n]^2}$, similar calculations as above yield the following tail bounds. For every $x>0$, we have
			\begin{align}	\label{tailbound}
			\p(Y > (A nR_n/2) + x R_n) \leq  \exp(-x^2/2). 
			\end{align}
		Since $M/2 = AG_n R_n \geq (AnR_n/2) + G_n R_n$, it follows from \eqref{tailbound} that 
			\begin{align}\label{tbfinal}
			\p(Y > M/2) \leq \exp(-G_{n}^{2}/2) \leq \exp(-n^2 g_n/2).
			\end{align}
				
				It now follows from \eqref{eq92}, \eqref{eq001} and \eqref{tbfinal} that 
				\begin{align}\label{lasteq}
				\p(\lt \geq 4nT) \leq 6 \exp \left\{ \frac{-\ep n^2}{2} \log \lb \frac{B}{R_{n}^{1/n}} \cdot \lb \frac{n}{T}\rb^{1-4\ep} \rb \right\}.
				\end{align}
				We have from \eqref{momrel} that $\lnn \geq R_{n}^{1/n}$ and now the desired result follows from \eqref{lasteq}.
				
\vskip .2cm
Now suppose that there exists $v \in \re^2$ such that  supp$(\mu) \subseteq \re v$,  without loss of generality we may assume that $v =(1,0)$. In this case $\mu$ considered as a measure on $\re$ satisfies \eqref{a1} and almost surely for every $(x,y) \in \re^2$, we have $X(x,y) = X(x,0)$ and hence the zero set of $X$ is $\cup_{t: X(t,0)=0}\{(t,y): y\in \re\}$. If we let $\nt$ denote the zero count of $X(\cdot, 0)$ in $[0,T]$, we have $\lt \geq nT$ iff $\nt \geq n$. The result now follows by using the estimates for $\p(\nt \geq n)$ obtained in Theorem \ref{thmonedim}.

\section{Consequences of Theorems \ref{thmonedim} $\&$ \ref{thmtwodim}}\label{seccons}

\subsection{One dimension} In this section, we deduce from Theorem \ref{thmonedim} the tail estimates for the zero count given in Table \ref{table1}.
\begin{enumerate}[label={\arabic*.},align=left,leftmargin=*,widest={7}]
	\item Let $\mu$ be compactly supported, say supp$(\mu) \subset [-q,q]$ for some $q>1$. Then $D_n \leq q^{2n}$ and  hence $D_{n}^{1/n} \leq q^2$. Taking $\ep = 1/4$ in \eqref{mainestimate0}, we conclude that there  exists $c,C>0$ such that whenever $T \geq 1$  and $n \geq CT $ we have
	\begin{align}\label{case1}
	\p(\nt \geq n)& \lesssim \exp \lb- cn^2 \log \left( n/\deltat \right) \rb.
	\end{align}
	\textit{Moment bounds.} We use \eqref{case1} to get  bounds for the moments of $\nt$, for $m \in \nat$ we have
	\begin{equation}
	\begin{aligned}\label{momentbounds}
	\e[\nt^m] & \lesssim (CT)^m + \sum_{n \geq CT} n^m e^{-cn^2},\\
	& \leq (CT)^m + (\sup_{x\geq 0} x^m e^{-c x^2/2}) \sum_{n \geq 1} e^{-cn^2/2},\\
	& \lesssim (CT)^m + (m/c)^{m/2} \leq (\tilde{c}(T \vee \sqrt{m}))^{m}.		
	\end{aligned}
	\end{equation}
	\item Let $\mu$ be such that  $C_n \lesssim n^{\alpha n }$ for some $\alpha>0$, then $D_{n}^{1/n} \lesssim n^{\alpha}$. 
	We consider the cases $\alpha <1$ and $\alpha \geq 1$ separately and make the following conclusions  from \eqref{mainestimate0}. 
	\begin{itemize}[align=left,leftmargin=*,widest={7}]
		\item Let $\alpha <1$, then for every $\kappa \in (0, 1-\alpha)$ there is  $c_{\kappa} >0$ such that whenever $ \deltat \leq n^{\kappa} $, we have
		\begin{align}\label{eqqq}
		\p(\nt \geq n) \lesssim  \exp({-c_{\kappa} n^2  \log n}). 
		\end{align}
		\textit{Justification for \eqref{eqqq}.} Let $\kappa',\kappa''$ be such $\kappa < \kappa' < \kappa'' < (1-\alpha)$, then we can choose $\ep \in (0,(1-\alpha)/2)$ such that $(1-\alpha - 2\ep)/(1- 2\ep) = \kappa''$ and hence for this choice of $\ep$ and large enough $n$, we have
		\begin{align*}
		\lb \frac{B}{n^{\alpha}} \lb \frac{n}{T}\rb^{1-2\ep}\rb^{\frac{1}{1-2\ep}} = B^{\frac{1}{1-2\ep}} \frac{n^{\frac{1-\alpha - 2\ep}{1-2\ep}}}{\deltat} \geq \frac{n^{\kappa'}}{\deltat} \geq \frac{n^{\kappa'}}{n^{\kappa}} = n^{\kappa' - \kappa},
		\end{align*}
		and hence 
		\begin{align*}
		\frac{B}{n^{\alpha}} \lb \frac{n}{T}\rb^{1-2\ep} \geq n^{(\kappa' - \kappa)(1-2\ep)} \geq e.
		\end{align*}
		And if we choose $n$ large enough, the fact that $T \leq n^{\kappa}$ will imply that $T \leq b\lfloor \ep n \rfloor$ and hence all the conditions of Theorem \ref{thmonedim} are satisfied. Using $\ep$ chosen above in \eqref{mainestimate0} and letting $c_{\kappa} = \ep(1-2\ep)(\kappa'- \kappa)/2$, we get \eqref{eqqq}. Similar arguments can also be used to justify the tail bounds we get for $\nt$ in the other cases, namely \eqref{case3a} and \eqref{case4a}. 
		\vskip .2cm
		\textit{Moment bounds.} By a calculation similar to \eqref{momentbounds}, we conclude using \eqref{eqqq} that 
		\begin{align*}
		\e[\nn_{T}^{m}] \leq (c'_{\kappa}(T^{1/\kappa} \vee \sqrt{m}))^{m}.
		\end{align*}
		\vskip .3 cm
		\item Let $\alpha \geq 1$, then for every $\kappa >0$ there exists $c_{\kappa}>0$ such that  whenever $n$ is large enough and $\deltat \leq 1/n^{\alpha -1+ \kappa} $, we have
		\begin{align}\label{case3a}
		\p(\nt \geq n) &\lesssim  \exp \lb -c_{\kappa} n^2  \log \lb \frac{e}{\deltat n^{\alpha -1 + \kappa}}\rb \rb.
		\end{align}
		We write $[0,1]$ as a union of $\lceil n^{\alpha-1+\kappa} \rceil$ many subintervals $\{I_{\ell}\}$,  each  of length $n^{-(\alpha-1+\kappa)}$. Then if $\nn_1 \geq  n^{\alpha + \kappa}$, at least one of the intervals $I_{\ell}$ must contain more than $n$ zeros and hence a union bound gives
		\begin{align}\label{unionbound}
		\p(\nn_1 \geq  n\cdot n^{\alpha-1+\kappa}) \leq \sum_{\ell}\p(\nn_{n^{-(\alpha-1+\kappa)}} \geq n) \lesssim n^{\alpha-1+\kappa} \exp(-c_{\kappa}n^2),
		\end{align}
		and hence there exists  $\tilde{c}_{\kappa} >0$ such that 
		\begin{align}\label{case3}
		\p(\mathscr{N}_1 \geq  n) &\lesssim \exp({-\tilde{c}_{\kappa}~ n^{\frac{2}{\alpha+\kappa}}  }).
		\end{align}
		\textit{Moment bounds.}	Using \eqref{case3} we get the following moment bounds for $\nn_1$. There exists $c_{\kappa} >0$ such that  for every $m \in \nat$, we have
		\begin{align*}
		\e[\nn_{1}^{m}] \leq (c_{\kappa} m^{\frac{\alpha + \kappa}{2}})^m.
		\end{align*}
		
	\end{itemize}
	
	\item Let $\mu$ be such that $\log C_n \lesssim n^{1+\frac{1}{\gamma}}$ for some $\gamma >1/2$. 
	We conclude from \eqref{mainestimate0} that there exist constants $c, c' >0$ such that whenever $n$ is large enough and $T \leq e^{-c n^{1/\ga}}$, we have
	\begin{align}\label{case4a}
	\p(\nt \geq n) \lesssim \exp(-c' n^2),
	\end{align}
	and hence by partitioning $[0,1]$ into subintervals of length $e^{-c n^{1/\ga}}$, using the above estimate in each of them and a union bound as in \eqref{unionbound} implies there is $\tilde{c}>0$ such that 
	\begin{align} \label{case4b}
	\p(\nn_1 \geq  n) \lesssim \exp(-\tilde{c}(\log n)^{2\ga}). 
	\end{align}
	\textit{Moment bounds.} We use \eqref{case4b} to conclude that there is $c>0$ such that for every $m \in \nat$, we have 
	\begin{align*}
	\e[\nn_{1}^{m}] \leq \exp\lb cm^{1+\frac{1}{2\ga -1}}\rb.
	\end{align*}
\end{enumerate}

\subsection{Two dimensions}
We now illustrate how the overcrowding estimates for nodal length given in Table \ref{table2} are deduced from Theorem \ref{thmtwodim}. Since this analysis is similar to  the one dimensional case considered above, we do this only for the first example in Table \ref{table2}.

Let $\mu$ be compactly supported, say supp$(\mu) \subset q\mathbb{D}$ for some $q>1$. Then $\lnn \lesssim q$ and hence by taking $\ep = 1/8$ in \eqref{mainest3} and by letting $\ell = 4 nT$, we conclude that there are constants $c,C>0$ such that whenever $T\geq 1$ and $\ell \geq CT^2$ is large enough, we have
\begin{align*}
\p(\lt \geq \ell)& \lesssim  \exp \lb- \frac{c \ell^2}{T^2} \log \left( \ell/\deltat^2 \right) \rb.
\end{align*}
\textit{Moment bounds.}	We get moment bounds using the above tail estimates. For $m \in \nat$, 
\begin{align*}
\e[\lt^m] &= \int_{0}^{\infty} mx^{m-1} \p(\lt \geq x) dx
\lesssim (CT^2)^m + m \int_{CT^2}^{\infty} x^{m-1} e^{-c x^2/T^2} dx,\\
& \lesssim (CT^2)^m + m (T/\sqrt{c})^m \int_{0}^{\infty} y^{m-1} e^{-y^2} dy \leq (c'T)^m (T \vee \sqrt{m})^m.
\end{align*}

\section*{Acknowledgements}
\noindent This work was carried out during my Ph.D. under the guidance of  Manjunath Krishnapur. I thank him for suggesting me the questions considered in this paper, his patience and encouragement throughout this study and for being very generous with his time and  insights on the subject of this paper. I also thank Riddhipratim Basu for helpful discussions. 

	\bibliographystyle{plain}
	\bibliography{reference_oc}
	\end{document}